\documentclass[review]{elsarticle}

\journal{some journal}

\usepackage{graphicx}
\usepackage{epstopdf}

\usepackage[utf8]{inputenc} 

\usepackage{amsmath}
\usepackage{amstext}
\usepackage{amsfonts}
\usepackage{amsthm}
\usepackage{amssymb}
\usepackage[bf,SL,BF]{subfigure}
\usepackage{caption}

\usepackage{color}

\usepackage{changes}

\newcommand{\R}{{\mathbb{R}}}

\newcommand{\E}{{\mathbb{E}}}
\newcommand{\N}{{\mathbb{N}}}

\newcommand{\F}{{\mathcal{F}}} 
\renewcommand{\P}{{\mathbb{P}}} 

\newcommand{\diff}[1]{\,\mathrm{d}#1}

\newcommand{\triple}{{\vert\kern-0.25ex\vert\kern-0.25ex\vert}}

\theoremstyle{plain}

\newtheorem{definition}{Definition}[section]
\newtheorem{theorem}[definition]{Theorem}
\newtheorem{lemma}[definition]{Lemma}

\newtheorem{assumption}[definition]{Assumption}

\theoremstyle{definition}
\newtheorem{remark}[definition]{Remark}

\newtheorem{example}[definition]{Example}

\newcommand{\eproof}{\hfill $\Box$}

\begin{document}

\begin{frontmatter}

\title{Truncated Euler-Maruyama method for classical and time-changed non-autonomous stochastic differential equations}

\author[SHNUaddress]{Wei Liu}

\author[STRATHaddress]{Xuerong Mao}

\author[SHNUaddress]{Jingwen Tang \corref{mycorrespondingauthor}}
\cortext[mycorrespondingauthor]{Corresponding author}
\ead{j.w.tang@foxmail.com}
\author[EDaddress,OXaddress]{Yue Wu}

\address[SHNUaddress]{Department of Mathematics, Shanghai Normal University, Shanghai, 200234, China}
\address[STRATHaddress]{Department of Mathematics and Statistics, University of Strathclyde, Glasgow, G1 1XH, UK}
\address[EDaddress]{School of Engineering, University of Edinburgh, Edinburgh, EH9 3JW, UK}
\address[OXaddress]{Mathematical Institute, University of Oxford, Oxford, OX2 6GG, UK}

\begin{abstract}
The truncated Euler-Maruyama (EM) method is proposed to approximate a class of non-autonomous stochastic differential equations (SDEs) with the H\"older continuity in the temporal variable and the super-linear growth in the state variable. The strong convergence with the convergence rate is proved. Moreover, the strong convergence of the truncated EM method for a class of highly non-linear time-changed SDEs is studied.
\end{abstract}


\end{frontmatter}


\section{Introduction} \label{sec:intro}
Stochastic differential equations (SDEs) have broad applications in many areas such as finance, physics, chemistry and biology \cite{A2007,PB2010}. However, most SDEs do not have the explicit expressions of the true solutions. Therefore, the numerical methods and the rigorous numerical analyses of those methods become extremely important \cite{KP1992,MT2004}.
\par
In this paper, we investigate the numerical approximation to the solutions of a class of non-autonomous stochastic differential equations of the It\^o type
\begin{align*}
  \begin{split}
  \begin{cases}
    \diff{x(t)}= \mu(t, x(t)) \diff{t} + \sum_{r=1}^m \sigma^r\big(t,x(t)\big)
    \diff{W^r(t)}, \quad t \in [t_0,T],\\
    x(t_0)= x_0,
  \end{cases}
  \end{split}
\end{align*}
where the coefficients obey the H\"older continuity in the temporal variable and the super-linear growth condition in the state variable. The detailed mathematical descriptions can be found in Section \ref{mathpre}.
\par
For non-autonomous SDEs with the H\"older continuous temporal variable in the coefficients, the randomized techniques are used to construct the Euler type method \cite{PM2014} and the Milstein type method \cite{KW2018}. However, most papers that investigate non-autonomous SDEs only consider the global Lipschitz condition for the state variable. Thus, one aim of this paper is to study the non-autonomous SDEs whose coefficients may grow super-linearly in the state variable, known as highly non-linear SDEs.
\par
The classic Euler-Maruyama (EM) method has been proved divergent for highly non-linear SDEs \cite{HJK2011}. While bearing in mind the idea that explicit methods have their advantages in simple algorithm structure and relatively lower computational cost in the simulations of a large number of sample paths \cite{H2011}, the tamed Euler method \cite{HJK2012} and the truncated Euler-Maruyama method \cite{M2015} are developed to approximate the solutions of highly non-linear SDEs. Some other interesting works on explicit methods for highly non-linear SDEs are, for example, \cite{DKS2016,GLMY2017,HLM2018,HJ2015,LMY2018,S2013,WG2013,ZM2017,ZSL2018,ZWH2014} and the references therein. However, those explicit methods proposed to tackle the super-linearity in the state variable do not take the non-autonomous SDEs into consideration.
\par
When both the H\"older continuity in the temporal variable and the super-linearity in the state variable appear together in one SDE, few works have been done on the numerical approximation to its solution. To fill up this gap, we investigate the truncated Euler-Maruyama method for this type of SDEs in this paper.
\par
The time-changed SDEs, where the time variable $t$ is replaced by some stochastic process $E(t)$ (see Section \ref{sec:tcSDE} for the details), have attracted lots of attentions in recent years \cite{DS2017,MS2015,MS2004,NN2017,UHK2018,Wu2016,ZY2019}. Due to the change of the time, the solution to the time-changed SDE is understood as a subdiffusion process, which could be used to describe diffusion phenomena that move slower than the Brownian motion \cite{AKMS2017,MS2013}. Numerical approximations to such type of SDEs are also important, as the explicit forms of the true solutions are rarely obtained. Only recently, authors in \cite{JK2016} studied the classical EM method for a class of time-changed SDEs, both of whose drift and diffusion coefficients satisfy the global Lipschitz condition. To our best knowledge, \cite{JK2016} is the first paper to investigate the numerical approximation to time-changed SDEs by directly discretising the equations. More recently, the semi-implicit EM method was proposed in \cite{DL2019} to approximate some time-changed SDEs with the global Lipschitz condition on the drift coefficient being replaced by the one-sided Lipschitz condition. Both of those two works used the duality principle proposed in \cite{Kob2011}, which, briefly speaking, relates the time-changed SDEs to certain kind of SDEs (see Section \ref{sec:tcSDE} for more details). In \cite{JK2019}, the authors investigated the classical EM for a larger class of time-changed SDEs without the application of the duality principle, though the drift and diffusion coefficients still satisfy the global Lipschitz condition. All the three works \cite{DL2019,JK2019,JK2016} investigated either the $L^1$ or $L^2$ convergence.
\par
In this paper, the truncated EM is used to approximate a class of time-changed SDEs of the form
\begin{equation*}
dy(t) = \mu (E(t),y(t))dE(t) + \sigma (E(t),y(t))dW(E(t)).
\end{equation*}
To our best knowledge, this is the first work devoted to numerical approximations to time-changed SDEs,  whose drift and diffusion coefficients are allowed to grow super-linearly. Moreover, we consider the $L^{\overline{q}}$ convergence for any $\overline{q} \geq 2.$
\par
The main contributions of our work are as follows.
\begin{itemize}
\item The truncated Euler-Maruyama method, which is an explicit method, is proved to be convergent to SDEs with the H\"older continuity in the temporal variable and the super-linearity in the state variable.
\item The convergence rate of $\min\{\alpha,\gamma,\frac12-\varepsilon\}$ is given, where $\alpha$ and $\gamma$ are the H\"older continuity indexes in the drift and diffusion coefficients, and $\varepsilon>0$ could be arbitrarily small.
\item The strong convergence of the truncated EM method for a class of time-changed SDEs, whose coefficients can grow super-linearly, is proved.
\end{itemize}
The paper is constructed as follows. Section \ref{mathpre} briefly introduces the truncated Euler-Maruyama method and some useful lemmas. The strong convergence with the rate for classical SDEs is presented and proved in Section \ref{mainresult}. The truncated EM method for time-changed SDEs is discussed in \ref{sec:tcSDE}. Numerical examples are given in Section \ref{numsimu} to demonstrate the theoretical results.

\section{Mathematical preliminaries}\label{mathpre}

This section is divided into three parts. In Section \ref{notationandassumption}, the notations and assumptions are introduced. To keep the paper self-contained, the truncated EM method is briefed in Section \ref{tEMemthod}. Some useful lemmas are presented in Section \ref{usefullemmas}.

\subsection{Notations and assumptions}\label{notationandassumption}

Throughout this paper, unless otherwise specified, we let $(\Omega_W,\F^W,\P_W)$ be a complete probability space with a filtration $\{\F^W_t\}_{t\in [0,T]}$ satisfying the usual conditions (that is, it is right continuous and increasing while $\F^W_0$ contains all $\P_W$-null sets), and let $\E_W$ denote the probability expectation with respect to $\P$. If $x\in \R^d$, then $|x|$ is the Euclidean norm. Let $x^T$ denotes the transposition of $x$. Moreover, for two real numbers $a$ and $b$, we use $a\vee b=\max(a,b)$ and $a\wedge b=\min(a,b)$.\par
For $d,m \in \N$, let $W:[t_0,T]\times\Omega_W\rightarrow \R^m$ be a standard $\{\F^W_t\}_{t\in [t_0,T]}$-Wiener process. Moreover, let $x:[t_0,T]\times\Omega_W\rightarrow \R^d$ be an $\{\F^W_t\}_{t\in [t_0,T]}$-adapted stochastic process that is a solution to It\^o type stochastic differential equation
\par \noindent
\begin{align}
  \label{eq:SDE}
  \begin{split}
  \begin{cases}
    \diff{x(t)} = \mu(t, x(t)) \diff{t} + \sum_{r=1}^m \sigma^r\big(t,x(t)\big)
    \diff{W^r(t)}, \quad t \in [t_0,T],\\
    x(t_0) = x_0,
  \end{cases}
  \end{split}
\end{align}
where $\E_W |x_0|^p < \infty$ for any $p > 0$, the drift coefficient function $\mu:[t_0,T]\times\R^d\rightarrow \R^d$ and the diffusion coefficient function $\sigma^r:[t_0,T]\times\R^d \rightarrow \R^d$ for $r\in \{1,2,\cdots,m\}$.\par
We impose the following assumptions on the drift and diffusion coefficients.

\begin{assumption}
 \label{as:Polynomial growth}
Assume that there exist positive constants $\beta$ and $M$ such that
\begin{equation*}
|\mu(t,x)-\mu(t,y)|\vee |\sigma^r(t,x)-\sigma^r(t,y)|\leqslant M(1+|x|^{\beta}+ |y| ^{\beta}) |x-y|,
\end{equation*}
for all $t\in [t_0,T]$, any $x,y\in \R^d$ and any $r\in\{1,2,\cdots,m\}$.
\end{assumption}

It can be observed from Assumption \ref{as:Polynomial growth} that all $t\in [t_0,T]$, $r\in\{1,2,\cdots,m\}$ and $x \in \R^d$
\par \noindent
\begin{align}
\label{mu(u)}
|\mu(t,x)|\vee |\sigma^r(t,x)|\leq K |x|^{\beta+1},
\end{align}
where $$K=2M+\sup_{t_0\leq t\leq T}(|\mu(t,0)|+\max_{1\leq r\leq m}|\sigma^r(t,0)|).$$

\begin{assumption}
\label{as:Khasminskii 2}
Assume that there exists a pair of constants $q>2$ and $L_1>0$ such that
\begin{align*}
(x-y)^T(\mu(t,x)-\mu(t,y))+\frac{q-1}{2}\sum_{r=1}^m|\sigma^r(t,x)-\sigma^r(t,y)|^2\leq L_1|x-y|^2,
\end{align*}
for all $t\in [t_0,T]$ and any $x,y\in \R^d$.
\end{assumption}

\begin{assumption}
\label{as:Khasminskii}
Assume that there exists a pair of constants $p>2$ and $L_2>0$ such that
\begin{align}
\label{k2}
    x^T \mu(t, x)+\frac{p-1}{2}
    \sum_{r=1}^m|\sigma^r(t, x) |^2 &\le L_2(1+|x|^2),
  \end{align}
for all $t\in [t_0,T]$ and any $x \in \R^d$.
\end{assumption}
\begin{remark}
It is clear that Assumption \ref{as:Khasminskii} may be derived from Assumption \ref{as:Khasminskii 2} but with more complicated coefficient in front of $|\sigma^r(t, x) |^2$. To keep the notation simple, we state Assumption \ref{as:Khasminskii} as a new assumption.
\end{remark}

\begin{assumption}
\label{as:superlinear growth}
Assume that there exist constants $\gamma\in (0,1]$, $\alpha\in(0,1]$, $K_1 > 0$ and $K_2 > 0$ such that
\begin{eqnarray*}
&&|\mu(t_1,x)-\mu(t_2,x)|\leq K_1(1+|x|^{\beta+1})|t_1-t_2|^\gamma,  \\
&&|\sigma^r(t_1,x)-\sigma^r(t_2,x))|\leq K_2(1+|x|^{\beta+1})|t_1-t_2|^\alpha,
\end{eqnarray*}
for all $t\in [t_0,T]$, any $x\in \R^d$ and any $r\in\{1,2,\cdots,m\}$, where the $\beta$ is the same as that in Assumption \ref{as:Polynomial growth}.
\end{assumption}

\subsection{The truncated Euler-Maruyama method for non-autonomous SDEs}\label{tEMemthod}
This part is to recall the truncated EM numerical scheme. To define the truncated EM numerical solutions with time $t$, we choose a strictly increasing continuous function $f:\R_+\to\R_+$ such that $f(u)\to \infty$ as $u\to \infty$ and
\par \noindent
\begin{equation*}
\sup_{t_0\leq t\leq T}\sup_{|x|\leq u}(|\mu(t,x)| \vee |\sigma(t,x)|)\leq f(u),\quad\forall u\geq1.
\end{equation*}
Denote by $f^{-1}$ the inverse function of $f$. It is clear that $f^{-1}$ is a strictly increasing continuous function from $[f(0),\infty)$ to $\R_+$. We also choose a constant $\hat{h}\geq 1\wedge |f(1)|$ and a strictly decreasing function $\kappa:(0,1]\to [|f(1)|,\infty)$ such that $$\lim_{\Delta\to0}\kappa(\Delta)=\infty,\qquad  \Delta^\frac14 \kappa(\Delta)\leq\hat{h},\quad\forall\Delta\in(0,1].$$\par
For a given step size $\Delta\in(0,1]$ let us define the truncated mapping $\pi_{\Delta}:\R^d\to\{x\in\R^d:|x|\leq f^{-1}(\kappa(\Delta))\}$ by
$$\pi_\Delta(x)=\bigg(|x|\wedge f^{-1}(\kappa(\Delta))\bigg)\frac{x}{|x|},$$
where we set $x/|x|=0$ when $x=0$.\par
Define the truncated functions by
\begin{align}
\label{definition}
\mu_\Delta(t,x)=\mu(t,\pi_\Delta(x)),\qquad \sigma_\Delta(t,x)=\sigma(t,\pi_\Delta(x)),
\end{align}
for $x\in\R^d$. It is easy to see that for any $t\in[t_0,T]$ and all $x\in \R^d$
\begin{align*}
|\mu_\Delta(t,x)|\vee|\sigma_\Delta(t,x)|\leq f(f^{-1}(\kappa(\Delta)))=\kappa(\Delta).
\end{align*}
The discrete-time truncated EM numerical solutions $x_\Delta(t_k)$, to approximate $x(t_k)$ for $t_k=k\Delta+t_0$, are formed by setting $x_\Delta(t_0)=x_0$ and computing
\begin{align*}
x_\Delta(t_{k+1})=x_\Delta(t_k)+\mu_\Delta(t_k,x_\Delta (t_k))\Delta+\sum_{r=1}^{m}\sigma_\Delta^r(t_k,x_\Delta(t_k))\Delta W_k^r,
\end{align*}
for $k=0,1,\cdots,N_\Delta$, where $N_\Delta$ is the integer part of $T/\Delta$ and we will set $t_{N_\Delta+1}=T$ while $\Delta W_k^r=W^r(t_{k+1})-W^r(t_k)$.
\par
To form the continuous versions of truncated EM numerical schemes, we define
\begin{align*}
\tau (t)=\sum_{k=0}^{N_\Delta} t_{k} I_{[t_{k},t_{k+1})}(t), ~~t\in[t_0,T].
\end{align*}
There are two versions of the continuous-time truncated EM solutions. The first one is defined by
\begin{align*}
\overline{x}_\Delta(t)=\sum_{k=0}^{N_\Delta} x_\Delta(t_k) I_{[t_{k},t_{k+1})}(t),
\end{align*}
which is a simple step process. The other one is defined by
\begin{align*}
x_\Delta(t)=x_0+\int_{t_0}^t \mu_\Delta(\tau(s),\overline{x}_\Delta(s))ds+\sum_{r=1}^m\int_{t_0}^t \sigma^r_\Delta(\tau(s),\overline{x}_\Delta(s))\diff W^r(s),
\end{align*}
which is continuous in $t \in [t_0,T]$.

\subsection{Some useful lemmas}\label{usefullemmas}
In this subsection, some lemmas that will be essential for the proof of the main result in Section \ref{mainresult} are presented. The proofs of these lemmas are either straightforward or can be found in references. Therefore, to focus our attention on the proof of the main result, those lemmas are stated without proofs.
\begin{lemma}
\label{lemma2}
Let Assumptions \ref{as:Polynomial growth} and \ref{as:Khasminskii} hold. The SDE (\ref{eq:SDE}) has a unique global solution $x(t)$. Moreover,
$$\sup_{t_0\leq t\leq T}\E_W|x(t)|^p<\infty.$$
\end{lemma}
The proof of the above lemma can be found in , for example, \cite{M2007}.

\begin{lemma}
\label{lemma1}
For any $\Delta\in (0,1]$ and any $\overline{p}>0$, we have
\begin{align*}
\E_W|x_\Delta(t)-\overline{x}_\Delta(t)|^{\overline{p}}\leq C_{\overline{p}}\Delta^{\frac{\overline{p}}{2}}(\kappa(\Delta))^{\overline{p}}, \quad\forall t\in[t_0,T],
\end{align*}
where $C_{\overline{p}}$ is a positive constant dependent only on $\overline{p}$. Consequently
\begin{align*}
\lim_{\Delta\to 0}\E_W|x_\Delta(t)-\overline{x}_\Delta(t)|^{\overline{p}}=0,\quad \forall t\in[t_0,T].
\end{align*}
\end{lemma}

\begin{lemma}
\label{lemma3}
Let Assumptions \ref{as:Polynomial growth} and \ref{as:Khasminskii} hold. Then
\begin{align*}
\sup_{0<\Delta\leq1}\sup_{t_0\leq t\leq T}\E_W|x_\Delta(t)|^p\leq C,
\end{align*}
where $C$ is a positive constant independent of $\Delta$.
\end{lemma}
From now on, the constants $C$, $C_1$, $C_2$, $C_3$, $C_{31}$ and $C_{32}$ stand for generic positive constants that are independent of $\Delta$ and their values may change between occurrences.

The proofs of Lemmas \ref{lemma1} and \ref{lemma3} follow straightforwardly from the proofs of Lemmas 3.1 and 3.2 in \cite{M2015}, by substituting $\mu_\Delta(t,\overline{x}_\Delta(s))$ and $\sigma_\Delta(t,\overline{x}_\Delta(s))$ for $\mu_\Delta(\overline{x}_\Delta(s))$ and $\sigma_\Delta(\overline{x}_\Delta(s))$, respectively.
\begin{remark}
From Lemma \ref{lemma3}, it is easily obtained that
$$\sup_{0<\Delta<1}\sup_{t_0\leq t\leq T}\E_W|\overline{x}_\Delta(t)|^p\leq C.$$
\end{remark}

\section{Main results on classical SDEs}\label{mainresult}

In this section, the strong convergence of the truncated EM method is proved and the convergence rate is given. The main theorem of this paper is as follows.

\begin{theorem}
\label{main theorem 2}
Let Assumptions \ref{as:Polynomial growth}, \ref{as:Khasminskii 2} and \ref{as:superlinear growth} hold. In addition, assume that (\ref{k2}) in Assumption \ref{as:Khasminskii} is true for any $p > 2$. Then for any $\overline{q} \geq 2$, $\Delta \in (0,1]$ and any $\varepsilon \in (0,1/4)$,
\begin{align*}
\sup_{t_0\leq t\leq T}\E_W|x(t)-x_\Delta(t)|^{\overline{q}}\leq C\Delta^{\min{(\gamma,\alpha,\frac12-\varepsilon)}{\overline{q}}},
\end{align*}
and
\begin{align*}
\sup_{t_0\leq t\leq T}\E_W|x(t)-\overline{x}_\Delta(t)|^{\overline{q}}\leq C\Delta^{\min{(\gamma,\alpha,\frac12-\varepsilon)}{\overline{q}}}.
\end{align*}
\end{theorem}
\begin{remark}
To obtain the results hold for any $\overline{q} \geq 2$ and arbitrarily $\varepsilon \in (0,1/4)$, Assumption \ref{as:Khasminskii} is strengthened by requiring  (\ref{k2}) to hold for any $p > 2$ instead of some $p > 2$ in Theorem \ref{main theorem 2}. In this circumstance, the $L_2$ in (\ref{k2})  is no longer dependent on $p$.
\end{remark}
To prove Theorem \ref{main theorem 2}, we show Theorem \ref{main theorem 1} firstly, in which the format of the convergence rate is a bit complicated. The proof of Theorem \ref{main theorem 2} is postponed after the proof of the following theorem.
\par
It should be noted that in Theorem \ref{main theorem 1} the Assumption \ref{as:Khasminskii} is not required to be strengthened compared with Theorem \ref{main theorem 2}.

\begin{theorem}
\label{main theorem 1}
Let Assumptions \ref{as:Polynomial growth}, \ref{as:Khasminskii 2}, \ref{as:Khasminskii} and \ref{as:superlinear growth} hold and assume that $p>(1+\beta)q$. Then, for any $\overline{q}\in[2,q)$ and $\Delta\in(0,1]$
\begin{align}
\label{assertion1}
\E_W|x(t)-x_\Delta(t)|^{\overline{q}} \leq C\bigg(\left(f^{-1}(\kappa(\Delta))\right)^{[(1+\beta)\overline{q}-p]/p}+\Delta^{\overline{q}/2}\left(\kappa(\Delta)\right)^{\overline{q}}+\Delta^{\gamma \overline{q}}+\Delta^{\alpha\overline{q}}\bigg),
\end{align}
and
\begin{align}
\label{assertion2}
\E_W|x(t)-\overline{x}_\Delta(t)|^{\overline{q}} \leq C\bigg(\left(f^{-1}(\kappa(\Delta))\right)^{[(1+\beta)\overline{q}-p]/p}+\Delta^{\overline{q}/2}\left(\kappa(\Delta)\right)^{\overline{q}}+\Delta^{\gamma \overline{q}}+\Delta^{\alpha\overline{q}}\bigg).
\end{align}
\end{theorem}

\begin{proof}
Fix $\overline{q}=[2,q)$ and $\Delta\in(0,1]$ arbitrarily. Let $e_\Delta(t)=x(t)-x_\Delta(t)$ for $t\in[t_0,T]$. By the It\^o formula, we have for any $t_0\leq t\leq T$,
\begin{eqnarray}
\label{Ito}
\E_W|e_\Delta(t)|^{\overline{q}}
&\leq&\E_W\int_{t_0}^t\overline{q}|e_\Delta(s)|^{\overline{q}-2}\bigg(e_\Delta^T(s)[\mu(s,x(s))-\mu_\Delta(\tau(s),\overline{x}_\Delta(s))]\nonumber\\
&&+\frac{\overline{q}-1}{2}\sum_{r=1}^{m}|\sigma^r(s,x(s))-\sigma_\Delta^r(\tau(s),\overline{x}_\Delta(s))|^2\bigg)ds.
\end{eqnarray}

By the Young inequality $2ab\leq\varepsilon a^2+b^2/\varepsilon$ for any $a,b\geq0$ and $\varepsilon$ arbitrary, choosing $\varepsilon=(q-\overline{q})/(\overline{q}-1)$ leads to
\begin{eqnarray*}
&&\frac{\overline{q}-1}{2}\sum_{r=1}^{m}|\sigma^r(s,x(s))-\sigma_\Delta^r(\tau(s),\overline{x}_\Delta(s))|^2\\
&\leq&\frac{\overline{q}-1}{2}\sum_{r=1}^{m}\bigg((1+\frac{q-\overline{q}}{\overline{q}-1})|\sigma^r(s,x(s))-\sigma^r(s,x_\Delta(s))|^2\\
&&+(1+\frac{\overline{q}-1}{q-\overline{q}})|\sigma^r(s,x_\Delta(s))-\sigma_\Delta^r(\tau(s),\overline{x}_\Delta(s))|^2\bigg)\\
&=&\frac{q-1}{2}\sum_{r=1}^{m}|\sigma^r(s,x(s))-\sigma^r(s,x_\Delta(s))|^2\\
&&+\frac{(\overline{q}-1)(q-1)}{2(q-\overline{q})}\sum_{r=1}^m|\sigma^r(s,x_\Delta(s))-\sigma_\Delta^r(\tau(s),\overline{x}_\Delta(s))|^2.
\end{eqnarray*}

We can get from(\ref{Ito}) that
\begin{eqnarray*}
&&\E_W|e_\Delta(t)|^{\overline{q}}\\
&\leq&\E_W \int_{t_0}^t \overline{q}|e_\Delta(s)|^{\overline{q}-2}\bigg(e_\Delta^T(s)[\mu(s,x(s))-\mu(s,x_\Delta(s))]\\
&&+\frac{p-1}{2}\sum_{r=1}^m[\sigma^r(s,x(s))-\sigma^r(s,x_\Delta(s))]^2\bigg)\diff s\\
&&+\E_W\int_{t_0}^t\overline{q}|e_\Delta(s)|^{\overline{q}-2}e^T_\Delta(t)[\mu(s,x_\Delta(s))-\mu(\tau(s),x_\Delta(s))]\diff s\\
&&+\E_W\int_{t_0}^t\overline{q}|e_\Delta(s)|^{\overline{q}-2}e^T_\Delta(t)[\mu(\tau(s),x_\Delta(s))-\mu_\Delta(\tau(s),\overline{x}_\Delta(s))]\diff s\\
&&+\E_W\int_{t_0}^t\overline{q}|e_\Delta(s)|^{\overline{q}-2}\frac{(\overline{q}-1)(q-1)}{(q-\overline{q})}
\sum_{r=1}^m|\sigma^r(s,x_\Delta(s))-\sigma^r(\tau(s),x_\Delta(s))|^2\diff s\\
&&+\E_W\int_{t_0}^t\overline{q}|e_\Delta(s)|^{\overline{q}-2}\frac{(\overline{q}-1)(q-1)}{(q-\overline{q})}
\sum_{r=1}^m|\sigma^r(\tau(s),x_\Delta(s))-\sigma_\Delta^r(\tau(s),\overline{x}_\Delta(s))|^2\diff s.
\end{eqnarray*}

This implies
\begin{align*}
\E_W|e_\Delta(t)|^{\overline{q}}\leq I_1+I_2+I_3,
\end{align*}
where
\begin{eqnarray*}
I_1&=&\E_W \int_{t_0}^t \overline{q}|e_\Delta(s)|^{\overline{q}-2}\bigg(e_\Delta^T(s)[\mu(s,x(s))-\mu(s,x_\Delta(s))]\\
&&+\frac{p-1}{2}\sum_{r=1}^m[\sigma^r(s,x(s))-\sigma^r(s,x_\Delta(s))]^2\bigg)\diff s,
\end{eqnarray*}

\begin{eqnarray*}
I_2&=&\E_W\int_{t_0}^t\overline{q}|e_\Delta(s)|^{\overline{q}-2}\bigg(e^T_\Delta(s)[\mu(s,x_\Delta(s))-\mu(\tau(s),x_\Delta(s))]\\
&&+\frac{(\overline{q}-1)(q-1)}{(q-\overline{q})}
\sum_{r=1}^m|\sigma^r(s,x_\Delta(s))-\sigma^r(\tau(s),x_\Delta(s))|^2\bigg)\diff s,
\end{eqnarray*}
and
\begin{eqnarray*}
I_3&=&\E_W\int_{t_0}^t\overline{q}|e_\Delta(s)|^{\overline{q}-2}\bigg(e^T_\Delta(s)[\mu(\tau(s),x_\Delta(s))-\mu_\Delta(\tau(s),\overline{x}_\Delta(s))]\\
&&+\frac{(\overline{q}-1)(q-1)}{(q-\overline{q})}
\sum_{r=1}^m|\sigma^r(\tau(s),x_\Delta(s))-\sigma_\Delta^r(\tau(s),\overline{x}_\Delta(s))|^2\bigg)\diff s.
\end{eqnarray*}
By Assumption~\ref{as:Khasminskii 2}, we have
\begin{align}
\label{I1}
I_1 \leq C_1 \E_W \int_{t_0}^t |e_\Delta(s)|^{\overline{q}}\diff s,
\end{align}
where $C_1=K_2\overline{q}.$
Using the Young inequality and Assumption~\ref{as:superlinear growth}, we can derive
\begin{eqnarray*}
I_2&\leq& \E_W\int_{t_0}^t\overline{q}|e_\Delta(s)|^{\overline{q}-2}\bigg(\frac12 |e_\Delta(s)|^2+\frac12|\mu(s,x_\Delta(s))-\mu(\tau(s),x_\Delta(s))|^2\\
&&+\frac{(\overline{q}-1)(q-1)}{(q-\overline{q})}
\sum_{r=1}^m|\sigma^r(s,x_\Delta(s))-\sigma^r(\tau(s),x_\Delta(s))|^2\bigg)\diff s\\
&\leq& C_2\bigg(\E_W \int_{t_0}^t |e_\Delta(s)|^{\overline{q}}ds+\E_W \int_{t_0}^t|\mu(s,x_\Delta(s))-\mu(\tau(s),x_\Delta(s))|^{\overline{q}}\diff s\\
&&+\frac{2(\overline{q}-1)(q-1)}{(q-\overline{q})}
\sum_{r=1}^m\E_W \int_{t_0}^t |\sigma^r(s,x_\Delta(s))-\sigma^r(\tau(s),x_\Delta(s))|^{\overline{q}}\diff s\bigg)\\
&\leq& C_2 \bigg(\E_W\int_{t_0}^t|e_\Delta(s)|^{\overline{q}}ds+\E_W\int_{t_0}^t K_1^{\overline{q}}(1+|x_\Delta(s)|^{\beta_1 \overline{q}})\Delta^{\gamma\overline{q}}\diff s\\
&&+\E_W\int_{t_0}^t K_2^{\overline{q}}(1+|x_\Delta(s)|^{\beta_2 \overline{q}})\Delta^{\alpha\overline{q}}\diff s\bigg).
\end{eqnarray*}
Then by Lemma~\ref{lemma2}, we obtain
\begin{align}
\label{I2}
I_2\leq C_2\bigg(\E_W\int_{t_0}^t|e_\Delta(s)|^{\overline{q}}\diff s+\Delta^{\gamma \overline{q}}+\Delta^{\alpha \overline{q}}\bigg).
\end{align}
Rearranging $I_3$ gives
\begin{eqnarray}
\label{I3}
I_3&\leq&\E_W\int_{t_0}^t\overline{q}|e_\Delta(s)|^{\overline{q}-2}\bigg(e^T_\Delta(t)[\mu(\tau(s),x_\Delta(s))-\mu(\tau(s),\overline{x}_\Delta(s))]\nonumber \\
&&+\frac{2(\overline{q}-1)(q-1)}{(q-\overline{q})}
\sum_{r=1}^m|\sigma^r(\tau(s),x_\Delta(s))-\sigma^r(\tau(s),\overline{x}_\Delta(s))|^2\bigg)\diff s \nonumber \\
&&+\E_W\int_{t_0}^t\overline{q}|e_\Delta(s)|^{\overline{q}-2}\bigg(e^T_\Delta(t)[\mu(\tau(s),\overline{x}_\Delta(s))-\mu_\Delta(\tau(s),\overline{x}_\Delta(s))]\nonumber \\
&&+\frac{2(\overline{q}-1)(q-1)}{(q-\overline{q})}
\sum_{r=1}^m|\sigma^r(\tau(s),\overline{x}_\Delta(s))-\sigma_\Delta^r(\tau(s),\overline{x}_\Delta(s))|^2\bigg)\diff s \nonumber \\
&:=&I_{31}+I_{32}.
\end{eqnarray}
By using the Young inequality and Assumption \ref{as:Polynomial growth}
we can show that
\begin{eqnarray*}
I_{31}&\leq&\E_W\int_{t_0}^t\overline{q}|e_\Delta(s)|^{\overline{q}-2}\bigg(\frac12|e^T_\Delta(t)|^2+
\frac12|\mu(\tau(s),x_\Delta(s))-\mu(\tau(s),\overline{x}_\Delta(s))|^2\\
&&+\frac{2(\overline{q}-1)(q-1)}{(q-\overline{q})}
\sum_{r=1}^m|\sigma^r(\tau(s),x_\Delta(s))-\sigma^r(\tau(s),\overline{x}_\Delta(s))|^2\bigg)\diff s\\
&\leq& C_{31}\bigg(\E_W\int_{t_0}^t|e_\Delta(s)|^{\overline{q}}\diff s
+\E_W\int_{t_0}^t|\mu(\tau(s),x_\Delta(s))-\mu(\tau(s),\overline{x}_\Delta(s))|^{\overline{q}}\\
&&+\sum_{r=1}^m|\sigma^r(\tau(s),x_\Delta(s))-\sigma^r(\tau(s),\overline{x}_\Delta(s))|^{\overline{q}}\diff s\bigg)\\
&\leq& C_{31}\bigg(\E_W\int_{t_0}^t|e_\Delta(s)|^{\overline{q}}\diff s\\
&&+M\E_W\int_{t_0}^t(1+|x_\Delta(s)|^{\beta\overline{q}}+|\overline{x}_\Delta(s)|^{\beta\overline{q}})
|x_\Delta(s)-\overline{x}_\Delta(s)|^{\overline{q}}\diff s\bigg).
\end{eqnarray*}
Then, by the H\"older inequality, Lemma~\ref{lemma2} and Lemma~\ref{lemma1}, we arrive at
\begin{eqnarray}
\label{I31}
I_{31}&\leq& C_{31}\bigg(\E_W\int_{t_0}^t|e_\Delta(s)|^{\overline{q}}\diff s+\int_{t_0}^t(\E_W|x_\Delta(s)-\overline{x}_\Delta(s)|^{p})^\frac{\overline{q}}{p}\diff s\bigg)\nonumber\\
&\leq& C_{31}\bigg(\E_W\int_{t_0}^t|e_\Delta(s)|^{\overline{q}}\diff s+\Delta^{\frac{\overline{q}}{2}}(\kappa(\Delta))^{\overline{q}}\diff s\bigg).
\end{eqnarray}
Similarly, we can show that
\begin{eqnarray*}
I_{32}&\leq& C_{32}\bigg(\E_W\int_{t_0}^t|e_\Delta(s)|^{\overline{q}}\diff s+\E_W\int_{t_0}^t|\mu(\tau(s),\overline{x}_\Delta(s))-\mu_\Delta(\tau(s),\overline{x}_\Delta(s))|^{\overline{q}}\\
&&+\sum_{r=1}^m|\sigma^r(\tau(s),\overline{x}_\Delta(s))-\sigma_\Delta^r(\tau(s),\overline{x}_\Delta(s))|^{\overline{q}}\diff s\bigg).
\end{eqnarray*}

Recalling the definition of truncated EM method (\ref{definition}) and Assumption~\ref{as:Polynomial growth} gives
\begin{eqnarray*}
I_{32}&\leq& C_{32}\bigg(\E_W\int_{t_0}^t|e_\Delta(s)|^{\overline{q}}\diff s+\E_W\int_{t_0}^t|\mu(\tau(s),\overline{x}_\Delta(s))-\mu(\tau(s),\pi_\Delta(\overline{x}_\Delta(s)))|^{\overline{q}}\\
&&+\sum_{r=1}^m|\sigma^r(\tau(s),\overline{x}_\Delta(s))-\sigma_\Delta^r(\tau(s),\pi_\Delta(\overline{x}_\Delta(s))|^{\overline{q}})\diff s\bigg)\\
&\leq& C_{32}\bigg(\E_W\int_{t_0}^t|e_\Delta(s)|^{\overline{q}}\diff s\\
&&+M\E_W\int_{t_0}^t(1+|\overline{x}_\Delta(s)|^{\beta\overline{q}}+|\pi_\Delta (\overline{x}_\Delta(s))|^{\beta\overline{q}})
|\overline{x}_\Delta(s)-\pi_\Delta (\overline{x}_\Delta(s))|^{\overline{q}}\diff s\bigg).
\end{eqnarray*}
By the H\"older inequality, we obtain
\begin{eqnarray}
\label{I32}
I_{32}&\leq& C_{32}\bigg(\E_W\int_{t_0}^t|e_\Delta(s)|^{\overline{q}}\diff s+\int_{t_0}^t[\E_W(1+|\overline{x}_\Delta(s)|^p+|\pi_\Delta(\overline{x}_\Delta(s))|^p)]^{\frac{\beta\overline{q}}{p}}\nonumber\\
&&\times(\E_W|\overline{x}_\Delta(s)-\pi_\Delta(\overline{x}_\Delta(s))|^{\frac{p\overline{q}}{p-\beta\overline{q}}})^{\frac{p-\beta\overline{q}}{p}}\diff s\bigg)\nonumber\\
&\leq& C_{32}\bigg(\E_W\int_{t_0}^t|e_\Delta(s)|^{\overline{q}}\diff s+\int_{t_0}^t (\E_W[I_{\{|\overline{x}_\Delta(s)|>f^{-1}(\kappa(\Delta))\}}|x_\Delta(s)|^{\frac{p\overline{q}}{p-\beta\overline{q}}}])^{\frac{p-\beta\overline{q}}{p}}
\diff s\bigg)\nonumber\\
&\leq& C_{32}\bigg(\E_W\int_{t_0}^t|e_\Delta(s)|^{\overline{q}}\diff s \nonumber \\
&&+\int_{t_0}^t([\P\{|\overline{x}_\Delta(s)|>f^{-1}(\kappa(\Delta))\}]^{\frac{p-\beta\overline{q}-\overline{q}}{p-\beta\overline{q}}}
[\E_W|\overline{x}_\Delta(s))|^p]^{\frac{\overline{q}}{p-\beta\overline{q}}})^{\frac{p-\beta\overline{q}}{p}}\diff s\bigg)\nonumber\\
&\leq& C_{32}\bigg(\E_W\int_{t_0}^t|e_\Delta(s)|^{\overline{q}}\diff s+\int_{t_0}^T \left(\frac{\E_W|\overline{x}_\Delta(s)|^p}{(f^{-1}(\kappa(\Delta)))^p}\right)^{\frac{p-\beta\overline{q}-\overline{q}}{p}}\diff s\bigg)\nonumber\\
&\leq& C_{32}\bigg(\E_W\int_{t_0}^t|e_\Delta(s)|^{\overline{q}}\diff s+(f^{-1}(\kappa(\Delta)))^{(\beta+1)\overline{q}-p}\bigg).
\end{eqnarray}
Substituting (\ref{I31}) and (\ref{I32}) into (\ref{I3}), we arrive at
\begin{align}
\label{I33}
I_3\leq C_3\bigg(\E_W\int_{t_0}^t|e_\Delta(s)|^{\overline{q}}\diff s+\Delta^{\frac{\overline{q}}{2}}(\kappa(\Delta))^{\overline{q}}+(f^{-1}(\kappa(\Delta)))^{(\beta+1)\overline{q}-p}\bigg).
\end{align}
Then (\ref{I1}), (\ref{I2}) and (\ref{I33}) together imply that
\begin{align*}
\E_W|e_\Delta(t)|^{\overline{q}}\leq C\bigg(\E_W\int_{t_0}^t|e_\Delta(s)|^{\overline{q}}\diff s+(f^{-1}(\kappa(\Delta)))^{(\beta+1)\overline{q}-p}+\Delta^{\frac{\overline{q}}{2}}(\kappa(\Delta))^{\overline{q}}+\Delta^{\gamma\overline{q}}+
\Delta^{\alpha\overline{q}}\bigg).
\end{align*}
An application of the Gronwall inequality yields that
\begin{align*}
\E_W|e_\Delta(t)|^{\overline{q}}\leq C\bigg((f^{-1}(\kappa(\Delta)))^{(\beta+1)\overline{q}-p}+\Delta^{\frac{\overline{q}}{2}}(\kappa(\Delta))^{\overline{q}}+\Delta^{\gamma\overline{q}}
+\Delta^{\alpha\overline{q}}\bigg),
\end{align*}
which is the required assertion (\ref{assertion1}). The other assertion (\ref{assertion2}) follows from (\ref{assertion1}) and Lemma \ref{lemma1}. Therefore, the proof is completed.
\end{proof}

Now, we are ready to give the proof of Theorem \ref{main theorem 2}.

\par \noindent
{\it \bf Proof of Theorem \ref{main theorem 2}}
\par
 Recalling (\ref{mu(u)}), we then define
\begin{align*}
f(u)=Ku^{\beta+1},~ u\geq1,
\end{align*}
which implies that
\begin{align*}
f^{-1}(u)=\left(\frac{u}{K}\right)^{\frac{1}{\beta+1}}.
\end{align*}
Let
\begin{align*}
\kappa(\Delta)=\Delta^{-\varepsilon} ~ \text{for} ~ \text{some}~\varepsilon\in(0,\frac14) ~\text{and} ~\hat{h}\geq1.
\end{align*}
Following Theorem \ref{main theorem 1}, we obtain
\begin{align}
\label{assertion3}
\E_W|x(t)-x_\Delta(t)|^{\overline{q}}\leq C\left(\Delta^\frac{\varepsilon(p-\beta\overline{q}-\overline{q})}{\beta+1}+\Delta^{\frac{\overline{q}(1-2\varepsilon)}{2}}+\Delta^{\gamma \overline{q}}+\Delta^{\alpha \overline{q}}\right),
\end{align}
and
\begin{align}
\label{assertion4}
\E_W|x(t)-\overline{x}_\Delta(t)|^{\overline{q}}\leq C\left(\Delta^\frac{\varepsilon(p-\beta\overline{q}-\overline{q})}{\beta+1}+\Delta^{\frac{\overline{q}(1-2\varepsilon)}{2}}+\Delta^{\gamma \overline{q}}+\Delta^{\alpha \overline{q}}\right).
\end{align}
Choosing $p$ sufficiently large for
$$\frac{\varepsilon(p-\beta\overline{q}-\overline{q})}{\beta+1}>\min({\gamma,\alpha,\frac12-\varepsilon})\overline{q},$$
we can draw the assertions from (\ref{assertion3}) and (\ref{assertion4}) immediately. \eproof

\section{Main results on time-changed SDEs}\label{sec:tcSDE}
This section is divided into two parts. In Section \ref{subsec:tcSDEpl}, mathematical preliminaries about time-changed SDEs are presented together with some useful lemmas. The result on the strong convergence of the truncated EM method is presented in Section \ref{subsec:tcSDEstroconv}.

\subsection{Mathematical preliminaries for time-changed SDEs}\label{subsec:tcSDEpl}
Let $D(t)$ be an RCLL increasing L\'evy process defined on a complete probability space $(\Omega_D , \F^D, \P_D)$ with a filtration $\left\{\F^D_t\right\}_{t \ge 0}$ satisfying the usual conditions. Let $\E_D$ denote the expectation under the probability measure $\P_D$. $D(t)$ is called subordinator starting from 0 if the Laplace transform is given by
\begin{equation*}
\E_D e^{-\lambda D(t)} = e^{-t \phi(\lambda)},
\end{equation*}
where the Laplace exponent is
\begin{equation*}
\phi (\lambda) = \int_0^{\infty} \left( 1 - e^{-\lambda x} \right) \nu (dx),
\end{equation*}
with $\int_0^{\infty} (x \wedge 1) \nu (dx) < \infty$. We focus on the case when the L\'evy measure $\nu$ is infinity, i.e. $\nu(0,\infty) = \infty$., which implies that $D(t)$ has strictly increasing paths with infinitely many jumps and excludes the compound Poisson subordinator.
\par
Let $E(t)$ be the inverse of $D(t)$, i.e.
\begin{equation*}
E(t) := \inf\{ u > 0; D(u) > t \}, ~ t \geq 0.
\end{equation*}
We call $E(t)$ an inverse subordinator.
\par
Assume that $W(t)$ and $D(t)$ are independent. Define the product probability space by
\begin{equation*}
(\Omega , \F, \P):= (\Omega_W \times \Omega_D, \F^W\otimes \F^D, \P_W \otimes \P_D).
\end{equation*}
Let $\E$ denote the expectation under the probability measure $\P$. It is clear that $\E(\cdot) = \E_D \E_W (\cdot) = \E_W \E_D (\cdot)$.
\par
In this section, we consider the following time-changed SDE
\begin{equation}\label{tcSDE}
dy(t) = \mu(E(t),y(t))dE(t) + \sigma(E(t),y(t))dW(E(t)), ~ t\in [0,T],
\end{equation}
with the initial value $y(0) = y_0$. Here, for the simplicity of the notation, we only consider the scale Wiener process $W$ (i.e. $m=1$ in Sections \ref{mathpre} and \ref{mainresult}).
\par
According to the duality principle in \cite{Kob2011}, the time-changed SDE \eqref{tcSDE} and the classical SDE of It\^o type
\begin{equation}\label{SDErttcSDE}
dx(t) = \mu(t,x(t))dt + \sigma(t,x(t))dW(t)
\end{equation}
have a deep connection. The next theorem states such a relation more precisely, which is borrowed from Theorem 4.2 in \cite{Kob2011}.
\begin{theorem}\label{lemma:equivxandy}
Suppose Assumptions \ref{as:Polynomial growth}, \ref{as:Khasminskii 2}, \ref{as:Khasminskii} and \ref{as:superlinear growth} hold. If $x(t)$ is the unique solution to the SDE \eqref{SDErttcSDE}, then the time-changed process $x(E(t))$, which is an $\F^W_{E(t)}$-semimartingale, is the unique solution to the time-changed SDE \eqref{tcSDE}. On the other hand, if $y(t)$ is the unique solution to the time-changed SDE \eqref{tcSDE}, then the process $y(D(t))$, which is an $\F^W_t$-semimartingale, is the unique solution to the SDE \eqref{SDErttcSDE}.
\end{theorem}
The plan to numerically approximate the time-changed SDE \eqref{tcSDE} in this section is as follows. Firstly, we discretize the inverse subordinator $E(t)$ to get $E_\Delta(t)$. Then the combination, $x_\Delta(E_\Delta(t))$, of the truncated EM solution to the SDE \eqref{SDErttcSDE}, $x_\Delta(t)$, and the discretized inverse subordinator, $E_\Delta(t)$, is used to approximate the solution to the time-changed SDE \eqref{tcSDE}.
\par
To approximate the $E(t)$ in a given time interval $[0,T]$, we follow the idea in \cite{GM2010}. Firstly, we simulate the path of $D(t)$ by $D_\Delta(t_i) = D_\Delta(t_{i-1} )+ \xi_i$ with $D(t_0) = 0$, where $\xi_i$ is independently identically sequence with $\xi_i = D(t_1)$ in distribution. The process is stopped when
\begin{equation*}
T \in [ D_\Delta (t_{n}), D_\Delta (t_{n+1})),
\end{equation*}
for some $n$. Then the approximate $E_\Delta(t)$ to $E(t)$ is generated by
\begin{equation}\label{findEht}
E_\Delta(t) = (\min\{n; D_\Delta (t_n) > t\} - 1)\Delta,
\end{equation}
for $t \in [0,T]$. It is easy to see
\begin{equation*}
E_\Delta(t) = i\Delta,~~~\text{when}~t \in   [ D_\Delta (t_{i}), D_\Delta (t_{i+1})).
\end{equation*}

The next lemma provides the approximation error of $E_\Delta(t)$ to $E(t)$, whose proof can be found in \cite{JK2016,Mag2009}.
\begin{lemma}\label{Eterror}
Let $E(t)$ be the inverse of a subordinator $D(t)$ with infinite L\'evy measure. Then for any $t \in [0,T]$
\begin{equation*}
E(t) - \Delta \leq E_\Delta(t) \leq E(t)~~~\text{a.s.}
\end{equation*}
\end{lemma}

The following lemma states that any inverse subordinator $E(t)$ with infinite L\'evy measure is known to have the exponential moment \cite{JK2016,MOW2011}.
\begin{lemma}\label{expmonfinite}
Let $E(t)$ be the inverse of a subordinator $D(t)$ with Laplace exponent $\phi$ and infinite L\'evy measure, then for any $C \in \R$ and $t \geq 0$,
\begin{equation*}
\E_D \left( e^{CE(t)} \right) < \infty.
\end{equation*}
\end{lemma}

We also need the continuity of the solution to \eqref{SDErttcSDE} presented in the next lemma. The proof is not hard to obtain by using the standard approach (see for example \cite{M2007}).
\begin{lemma}\label{lemma:holdery}
Suppose that Assumptions \ref{as:Polynomial growth} and \ref{as:Khasminskii} hold. Then for any $q < p/(\beta+1)$ and $|t - s| < 1$, the solution to \eqref{SDErttcSDE} satisfies
\begin{equation*}
\E_W |x(t) - x(s)|^q \leq C_4 |t - s|^{q/2} e^{C_4t},
\end{equation*}
where $C_4$ is a constant independent of $t$ and $s$.
\end{lemma}

\subsection{Strong convergence of the truncated EM method for time-changed SDEs}\label{subsec:tcSDEstroconv}
Before the main result is presented, we make some remarks on the constant, $C$, in Theorem \ref{main theorem 2}. Since the main purpose of Theorem \ref{main theorem 2} is to show the convergence rate, we do not give the explicit form of the constant $C$. But it is not hard by going through the proof to see that $C(t):=C$ contains the time variable $t$ only in the form of $\exp (\text{some constant}\times t )$. This means that if we replace $t$ by $E(t)$, we have $\E_D(C(E(t))) < \infty$ by Lemma \ref{expmonfinite}.

\begin{theorem}\label{thm:timechangenum}
Let Assumptions \ref{as:Polynomial growth}, \ref{as:Khasminskii 2} and \ref{as:superlinear growth} hold. In addition, assume that (\ref{k2}) in Assumption \ref{as:Khasminskii} is true for any $p > 2$. Then the combination of the truncated Euler-Maruyama solution and the discretized inverse subordinator, i.e. $x_\Delta(E_\Delta(t))$, converges strongly to the solution of \eqref{tcSDE}
\begin{equation*}
\E \left| y(t) - x_\Delta(E_\Delta(t))\right|^{\overline{q}} \leq C_{tc} \Delta^{\min(\gamma,\alpha,\frac12-\varepsilon){\overline{q}}},
\end{equation*}
for any $\overline{q} \geq 2$, $\Delta \in (0,1]$, $\varepsilon \in (0,1/4)$ and $t \in [0,T]$, where $C_{tc}$ is constant independent from $\Delta$.
\end{theorem}

\begin{proof}
By Theorem \ref{lemma:equivxandy} and the elementary inequality, we have
\begin{align*}
&~~~~\E \left| y(t) - x_\Delta (E_\Delta(t))\right|^{\overline{q}} \nonumber \\
&= \E \left| x(E(t)) - x_\Delta (E_\Delta(t))\right|^{\overline{q}} \nonumber \\
&\leq  2^{{\overline{q}} - 1} \left( \E \left| x(E(t)) - x(E_\Delta(t))\right|^{\overline{q}} +   \E \left| x(E_\Delta(t) - x_\Delta (E_\Delta(t))\right|^{\overline{q}} \right).
\end{align*}
By Lemmas \ref{Eterror}, \ref{expmonfinite} and \ref{lemma:holdery},  we can see
\begin{equation} \label{yyerror}
\E \left| x(E(t)) - x(E_\Delta(t))\right|^{\overline{q}} \leq C_4 \Delta^{\overline{q}/2} \E_D \left( e^{C_4 E(t)} \right) \leq C_5 \Delta^{\overline{q}/2},
\end{equation}
where $C_5$ is a constant independent from $\Delta$.
By Lemma \ref{Eterror} and Theorem \ref{main theorem 2}, we obtain
\begin{equation}\label{yYerror}
\E \left| x(E_\Delta(t) - x_\Delta(E_\Delta(t))\right|^{\overline{q}} \leq \E_D(C) \Delta^{\min{(\gamma,\alpha,\frac12-\varepsilon)}{\overline{q}}}\leq  C_6 \Delta^{\min{(\gamma,\alpha,\frac12-\varepsilon)}{\overline{q}}},
\end{equation}
where $C_6$ is a constant independent from $\Delta$. Combining \eqref{yyerror} and \eqref{yYerror}, we have the required assertion.
\end{proof}

\section{Numerical Simulations}\label{numsimu}
This section is divided into two parts. The numerical simulations on SDEs are presented in Section \ref{numsimuSDE} and time-changed SDEs are displayed in Section \ref{numsimutcSDE}.
\subsection{Simulations for SDEs}\label{numsimuSDE}
Two examples with the different theoretical convergence rates are presented in this part. Computer simulations are conducted to verify the theoretical results.
\begin{example}
\label{example1}
Consider a scalar stochastic differential equation
\begin{align}
  \label{example:SDE}
  \begin{split}
  \begin{cases}
    \diff{x(t)} =\bigg([t(1-t)]^{\frac14}x^2(t)-2x^5(t)\bigg) \diff{t} + \bigg([t(1-t)]^{\frac14}x^2(t)\bigg)\diff W(t), \\
    x(t_0) =2,
  \end{cases}
  \end{split}
\end{align}
with $t_0=0$ and $T=1$.
\par
For any $q>2$, we can see
\begin{eqnarray*}
&&(x-y)^T(\mu(t,x)-\mu(t,y))+\frac{q-1}{2}|\sigma^r(t,x)-\sigma^r(t,y)|^2\\
&\leq&(x-y)^2\bigg([t(1-t)]^{\frac14}(x+y)-2(x^4+x^3y+x^2y^2+xy^3+y^4)\\
&&+\frac{q-1}{2}[t(1-t)]^{\frac{1}{2}}(x+y)^2\bigg).
\end{eqnarray*}
But
\begin{equation*}
-2(x^3y + xy^3) = -2xy(x^2 + y^2) \leq (x^2+ y^2)^2 = x^4 + y^4 +2x^2y^2.
\end{equation*}
Therefore,
\begin{eqnarray*}
&&(x-y)^T(\mu(t,x)-\mu(t,y))+\frac{q-1}{2}|\sigma^r(t,x)-\sigma^r(t,y)|^2\\
&\leq& (x-y)^2\bigg([t(1-t)]^{\frac14}(x+y)-x^4-y^4+(q-1)[t(1-t)]^{\frac{1}{2}}(x^2+y^2)\bigg)\\
&\leq& L_1 (x-y)^2,
\end{eqnarray*}
where the last inequality is due to the fact that polynomials with the negative coefficient for the highest order term can always be bounded from above.
This indicates that Assumption \ref{as:Khasminskii 2} holds.\par
In addition, for any $p>2$, we have
\begin{eqnarray*}
x^T\mu(t,x)+\frac{p-1}{2}|\sigma(t,x)|^2\leq x^3-2x^6+\frac{p-1}{2}|x|^4\leq K_1(1+|x|^2),
\end{eqnarray*}
which means that Assumption \ref{as:Khasminskii} is satisfied.
\par \noindent
Using the mean value theorem for the temporal variable, Assumptions \ref{as:Polynomial growth} and \ref{as:superlinear growth} are satisfied with $\alpha=\gamma=1/4$ and $\beta=4$. According to Theorem~\ref{main theorem 1}, we know that
\par \noindent
\begin{align*}
\E_W|x(t)-x_\Delta(t)|^{\overline{q}} \leq C\bigg((f^{-1}(\kappa(\Delta)))^{(5\overline{q}-p)/p}+\Delta^{\overline{q}/2}(\kappa(\Delta))^{\overline{q}}+\Delta^{ \overline{q}/4}\bigg),
\end{align*}
\par \noindent
and
\begin{align*}
\E_W|x(t)-\overline{x}_\Delta(t)|^{\overline{q}} \leq C\bigg((f^{-1}(\kappa(\Delta)))^{(5\overline{q}-p)/p}+\Delta^{\overline{q}/2}(\kappa(\Delta))^{\overline{q}}+\Delta^{ \overline{q}/4}\bigg).
\end{align*}
\par \noindent
Due to that
\par \noindent
\begin{align*}
\sup_{t_0\leq t\leq T}\sup_{|x|\leq u}(|\mu(t,x)\vee|\sigma(t,x)|)\leq 3u^5,\quad\forall u\geq1,
\end{align*}
\par \noindent
we choose $f(u)=3u^5$ and $\kappa(\Delta)=\Delta^{-\varepsilon}$, for any $\varepsilon\in(0,1/4)$. As a result, $f^{-1}(u)=(u/3)^{1/5}$ and $f^{-1}(\kappa(\Delta))=(\Delta^{-\varepsilon}/3)^{1/5}$. Choosing $p$ sufficiently large, we can get from Theorem ~\ref{main theorem 2} that
\begin{align*}
\sup_{0\leq t\leq 1}\E_W|x(t)-x_\Delta(t)|^{\overline{q}} \leq  C\Delta^{\overline{q}/4},
\end{align*}
and
\begin{align*}
\sup_{0\leq t\leq 1}\E_W|x(t)-\overline{x}_\Delta(t)|^{\overline{q}} \leq C\Delta^{\overline{q}/4},
\end{align*}
which imply that the convergence rate of truncated EM method for the SDE (\ref{example:SDE}) is $1/4$.\par
Let us compute the approximation of the mean square error. We run $M=1000$ independent trajectories for every different step sizes, $10^{-1}$, $10^{-2}$, $10^{-3}$, $10^{-8}$. Because it is hard to find the true solution for the SDE, the numerical solution with the step size $10^{-8}$ is regarded as the exact solution.

\begin{figure}[htbp]
\centering
\subfigure[Convergence rate of  Example \ref{example1}]
{
  \begin{minipage}{5cm}
  \label{a}
  \centering
  \includegraphics[scale=0.38]{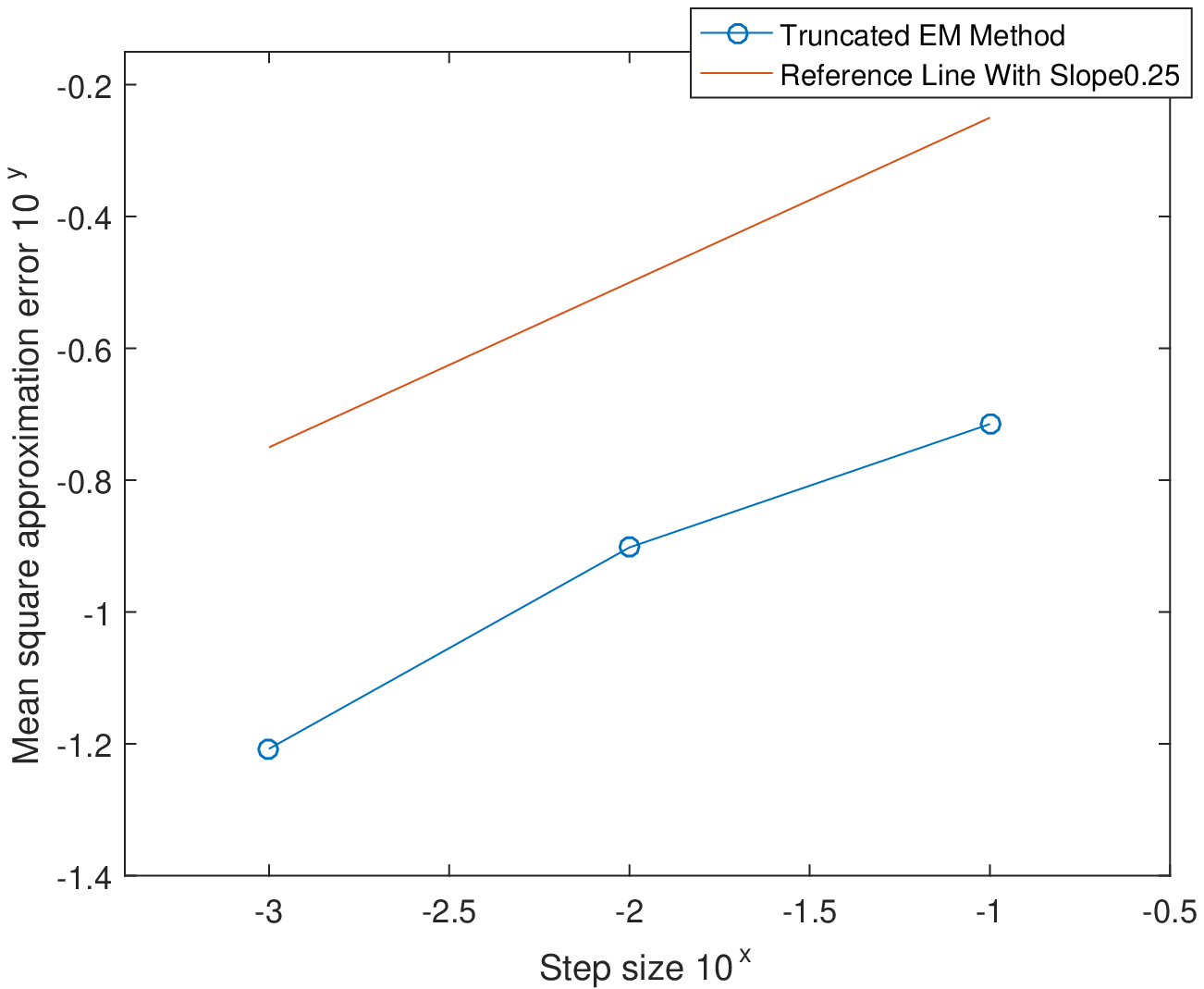}
  \end{minipage}
}
\subfigure[Convergence rate of Example \ref{example2}]
{
  \begin{minipage}{5cm}
  \label{b}
  \centering
  \includegraphics[scale=0.38]{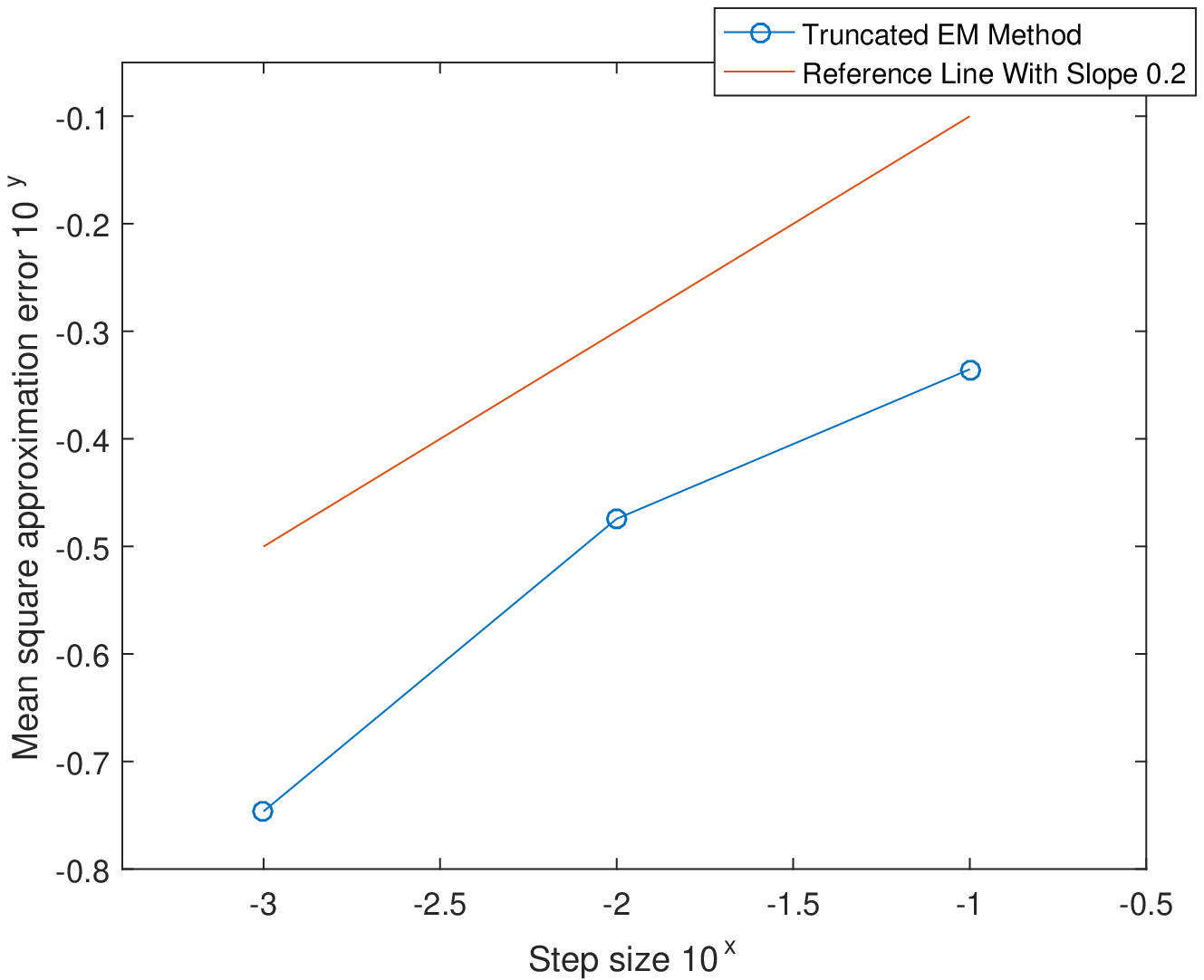}
  \end{minipage}
}
\caption{The $L^1$ errors between the exact solution and the numerical solutions for step sizes $\Delta=10^{-1},~10^{-2},~10^{-3}$.}
\end{figure}

By the linear regression, also shown in the Figure \ref{a}, the slope of the errors against the step sizes is approximately 0.24629, which is quite close to the theoretical result.
\end{example}

\begin{example}
\label{example2}
Consider the scalar stochastic differential equation
\begin{align}
  \label{example2:SDE}
  \begin{split}
  \begin{cases}
    \diff{x(t)} =\bigg([(t-1)(2-t)]^{\frac15}x^2(t)-2x^5(t)\bigg) \diff{t} + \bigg([(t-1)(2-t)]^{\frac25}x^2(t)\bigg)\diff W(t),\\
    x(t_0) =2,
  \end{cases}
  \end{split}
\end{align}
where $t_0=1$ and $T=2$. In the similar way as Example \ref{example1}, we can verify that Assumptions \ref{as:Khasminskii 2} and \ref{as:Khasminskii} hold.
\par
Moreover, the mean value theorem is used to verify that Assumptions \ref{as:Polynomial growth} and \ref{as:superlinear growth} are satisfied with $\alpha=2/5, \gamma=1/5$ and $\beta=4$.\par
We can get from Theorem \ref{main theorem 2} that
\begin{align*}
\sup_{1\leq t\leq 2}\E_W|x(t)-x_\Delta(t)|^{\overline{q}} \leq C\Delta^{\overline{q}/5},
\end{align*}
and
\begin{align*}
\sup_{1\leq t\leq 2}\E_W|x(t)-\overline{x}_\Delta(t)|^{\overline{q}} \leq C\Delta^{\overline{q}/5},
\end{align*}
which implies that the convergence rate of truncated EM method for the SDE (\ref{example2:SDE}) is $1/5$.
Simulation is conducted using the same strategy as that in Example \ref{example1}.
Using the linear regression, also seen in the figure \ref{b}, the slope of the errors against the step sizes is approximately 0.20550, which coincides with the theoretical result.
\end{example}

\subsection{Simulations for time-changed SDEs}\label{numsimutcSDE}

\begin{example}
A two-dimensional time-changed SDE
\begin{align}
  \label{expl:strcon}
  \begin{split}
  \begin{cases}
    \diff{y_1(t)} = -2y_1^4(t) \diff{t} + y_2^2(t) \diff W(t),\\
    \diff{y_2(t)} = -2y_2^4(t) \diff{t} + y_1^2(t) \diff W(t),
  \end{cases}
  \end{split}
\end{align}
is considered with the initial data $y_1(0) = 1$ and $y_2(0) = 2$.
\end{example}
\par
For a given step size $h$, one path of the numerical solution to \eqref{expl:strcon} is simulated in the following way.
\par
\noindent
{\bf Step 1.} The truncated EM method with the step size $\Delta$ is used to simulate the numerical solution, $x_\Delta (t_k)$, for $k=1,2,3,...$, to the duel SDE
\begin{align*}
  \begin{split}
  \begin{cases}
    \diff{x_1(t)} = -2x_1^4(t) \diff{t} + x_2^2(t) \diff W(t),\\
    \diff{x_2(t)} = -2x_2^4(t) \diff{t} + x_1^2(t) \diff W(t).
  \end{cases}
  \end{split}
\end{align*}
\par
\noindent
{\bf Step 2.} One path of the subordinator $D(t)$ is simulated with the same step size $\Delta$. (see for example \cite{JUMthesis2015}).
\par
\noindent
{\bf Step 3.} The $E_h(t)$ is found by using \eqref{findEht}.
\par
\noindent
{\bf Step 4.} The combination, $x_\Delta(E_h(t))$, is used to approximate the solution to \eqref{expl:strcon}.

\begin{figure}[htbp]
\centering
\subfigure[One path of $D(t)$]
{
  \begin{minipage}{5.7cm}
  \label{fig:pathDt}
  \centering
  \includegraphics[scale=0.44]{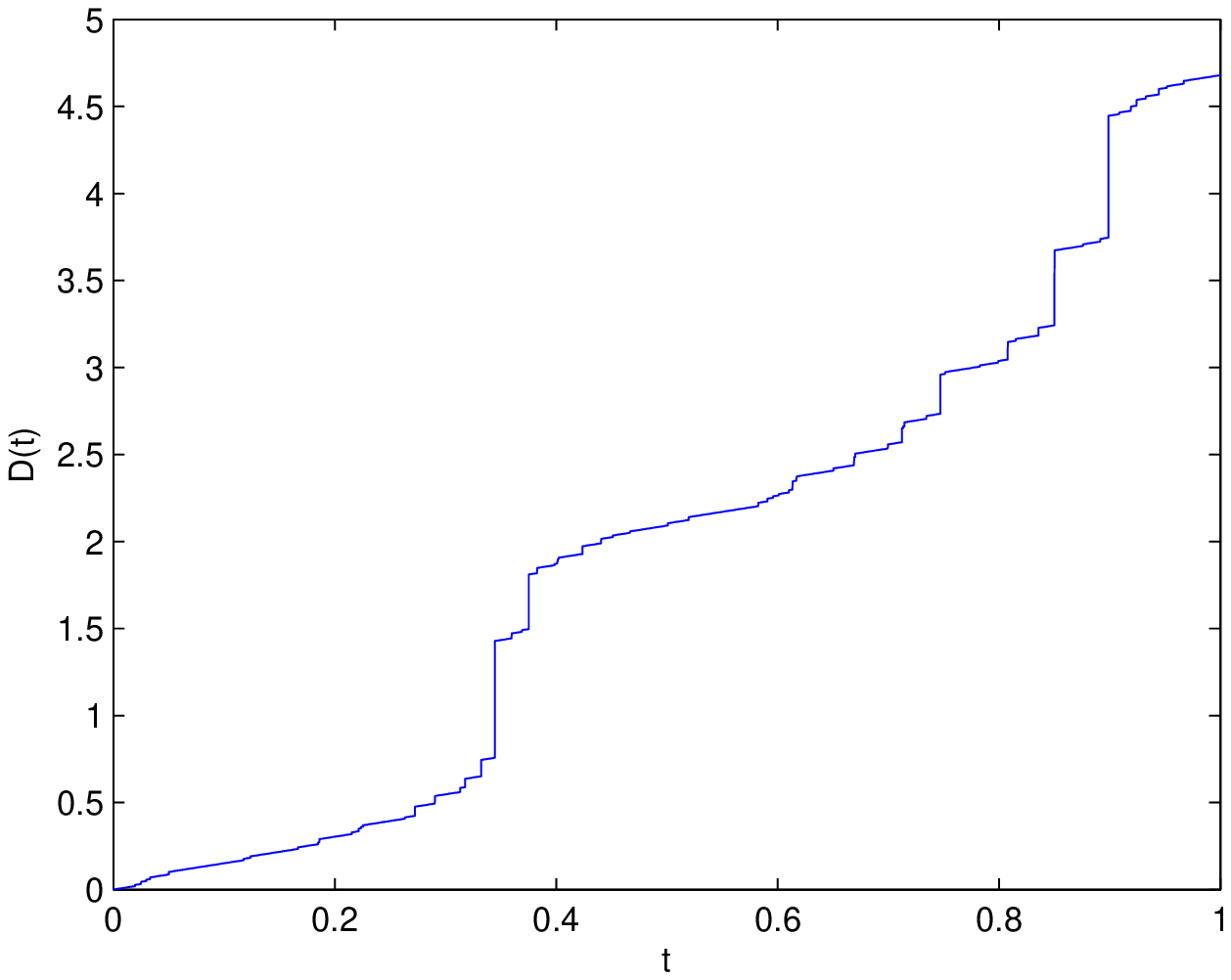}
  \end{minipage}
}
\subfigure[One path of $E(t)$]
{
  \begin{minipage}{5.7cm}
  \label{fig:pathEt}
  \centering
  \includegraphics[scale=0.44]{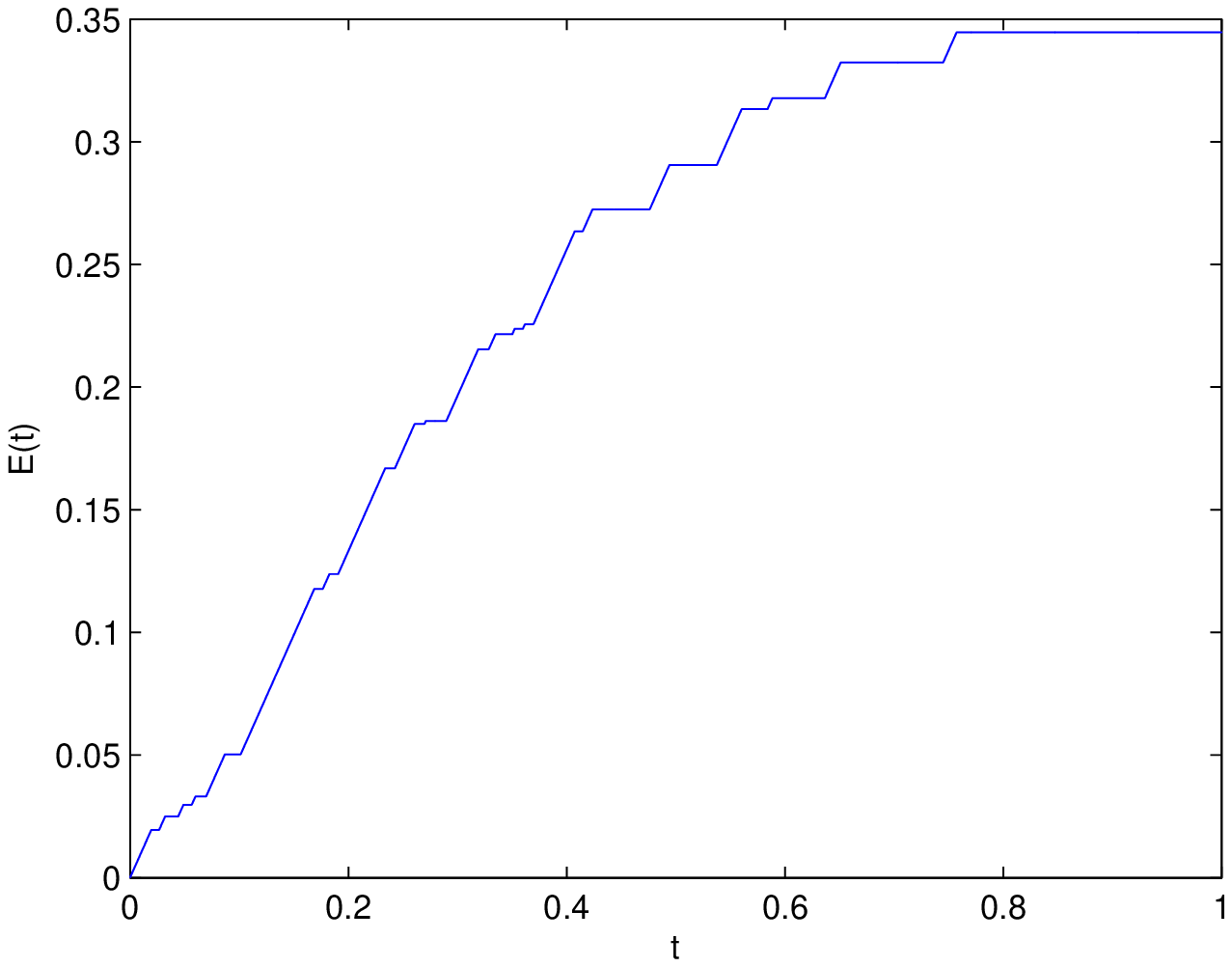}
  \end{minipage}
}
\subfigure[One path of $y_1(t)$]
{
  \begin{minipage}{5.7cm}
  \label{fig:pathy1}
  \centering
  \includegraphics[scale=0.44]{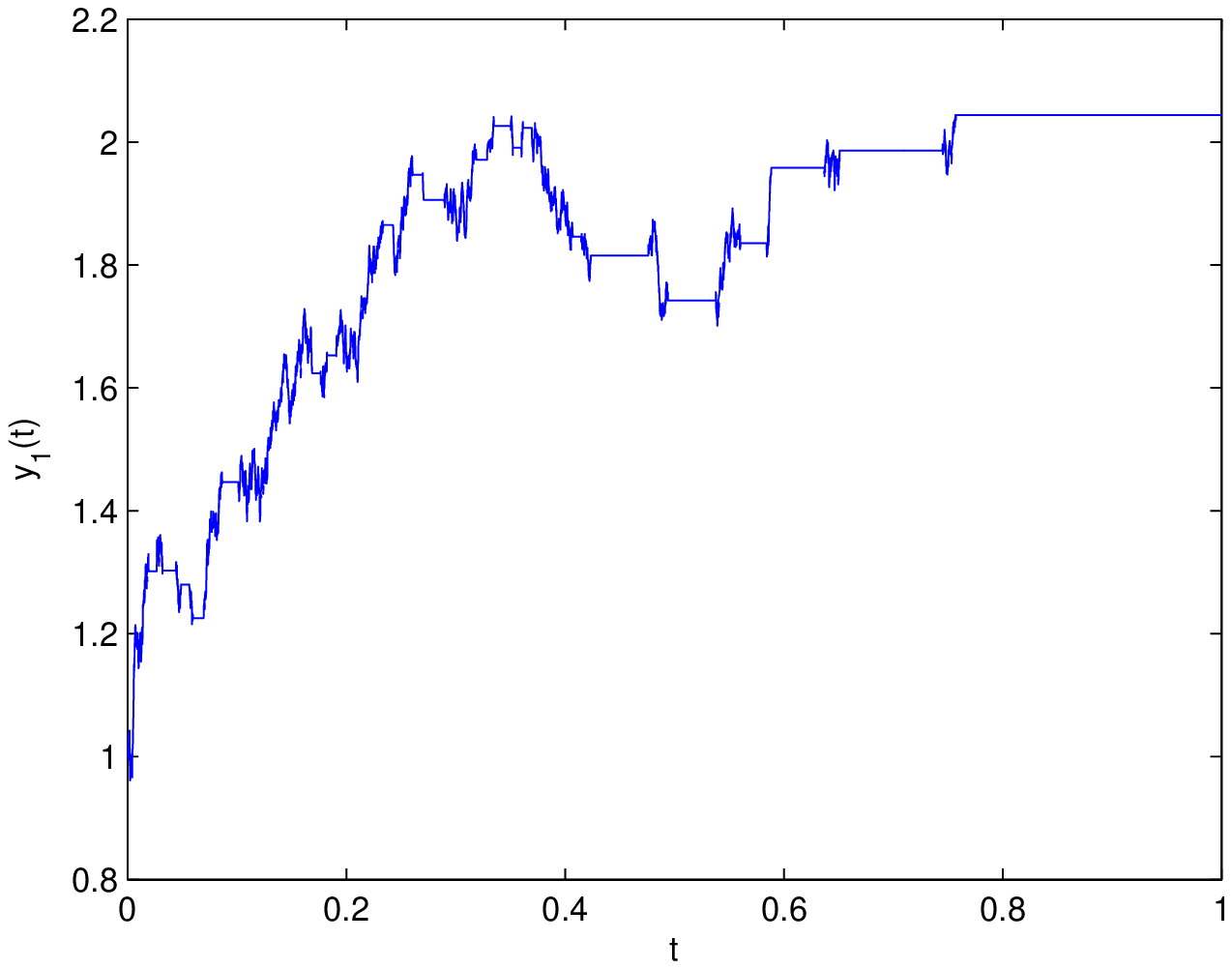}
  \end{minipage}
}
\subfigure[One path of $y_2(t)$]
{
  \begin{minipage}{5.7cm}
  \label{fig:pathy2}
  \centering
  \includegraphics[scale=0.44]{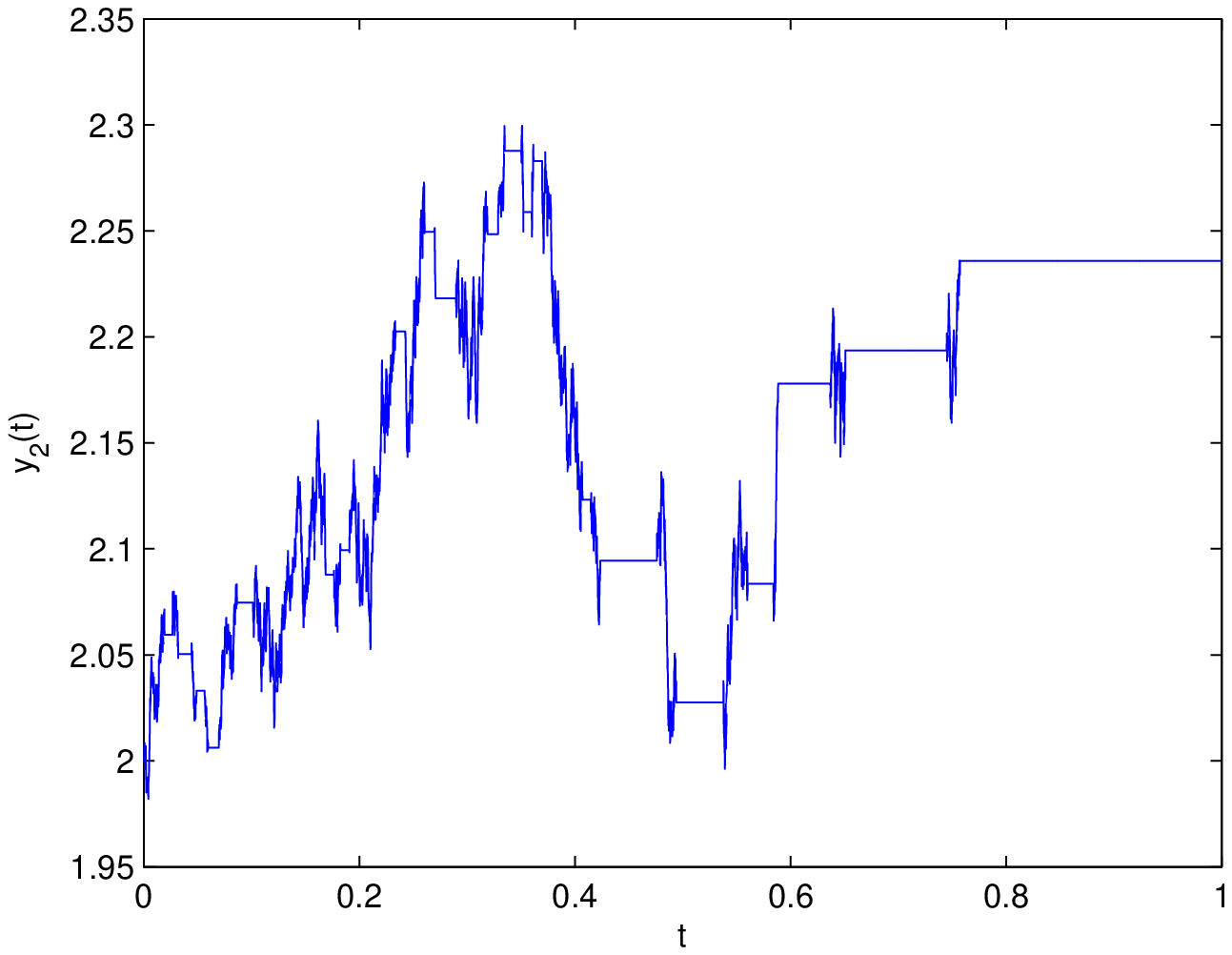}
  \end{minipage}
}
\caption{Numerical simulations of $D(t)$, $E(t)$, $y_1(t)$ and $y_2(t)$}
\end{figure}
For $t\in[0,1]$ and $\Delta=10^{-4}$, Figure \ref{fig:pathDt} shows one path of $D(t)$ and Figure \ref{fig:pathEt} displays one path of $E(t)$. Paths of $y_1(t)$ and $y_2(t)$ are plotted in Figures \ref{fig:pathy1} and \ref{fig:pathy2}, respectively.
\par
Now we demonstrate the strong convergence rate. Since the explicit solution ia hard to obtain, we treat the numerical solution with $\Delta = 10^{-8}$ as the true solution. One hundred samples are used to compute the strong convergence with the step sizes $10^{-2}$, $10^{-3}$ and $10^{-4}$. We pick up $\epsilon=0.01$, by Theorem \ref{thm:timechangenum} a strong convergence rate that is closed to 0.5 is expected. Figure \ref{fig:strcov} illustrate such a  convergence rate.
\begin{figure}
\begin{center}
  \includegraphics[scale=0.44]{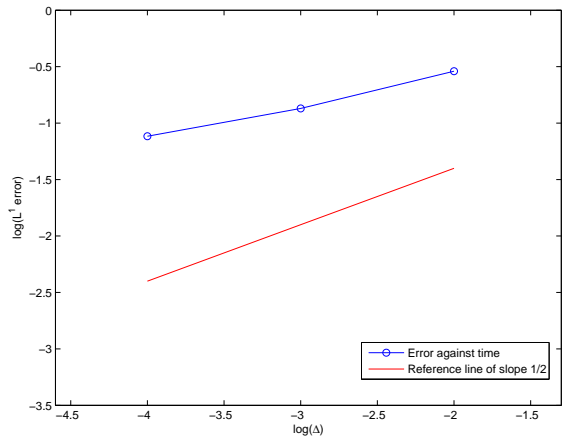}
\end{center}
\caption{Blue line: Loglog plot of the strong $L^1$ error against the step size. Red Line: The reference line with the slope of 1/2.}
\label{fig:strcov}
\end{figure}
\section{Conclusion} \label{sec:conclu}
In this paper, we apply the truncated EM method for a class of non-autonomous classical SDEs with the H\"older continuity in the temporal variable and the super-linear growth in the state variable. The strong convergence with the rate is proved.
\par
In addition, the results on the classical SDEs are used to prove that the truncated EM method can also work well for a class of highly non-linear time-changed SDEs. Such a result provides a trusted numerical method for a much larger class of time-changed SDEs than those in existing works.

\section*{Acknowledgement}
Wei Liu is financially supported by the National Natural Science Foundation of China (11701378, 11871343), “Chenguang Program” supported by both Shanghai Education Development Foundation and Shanghai Municipal Education Commission (16CG50), and Shanghai Gaofeng \& Gaoyuan Project for University Academic Program Development.
\par
Xuerong Mao would like to thank the EPSRC (EP/K503174/1), the Royal Society (Wolfson Research Merit Award WM160014), the Royal Society and the
Newton Fund (NA160317, Royal Society-Newton Advanced Fellowship) and the Ministry of Education (MOE) of China (MS2014DHDX020), for their financial support.
\par
Yue Wu is financially supported by EPRSC EP/R041431/1.

\end{document}